\documentclass[letterpaper, 10pt, journal, final]{IEEEtran}  
\IEEEoverridecommandlockouts 

\usepackage{mathtools}
\usepackage{hyperref}
\usepackage{graphicx} 
\usepackage{amsmath} 
\usepackage{amssymb}  
\usepackage{multirow}
\usepackage{color}
\usepackage{algorithm, algorithmicx, algpseudocode}

\newcommand{\R}{{\mathbb R}}

\newcommand{\Rnn}{{\mathbb R}_{\ge 0}}
\newcommand{\Rp}{{\mathbb R}_{> 0}}
\newcommand{\C}{{\mathbb C}}

\newcommand{\cB}{{\mathcal B}}
\newcommand{\cC}{{\mathcal C}}
\newcommand{\cD}{{\mathcal D}}
\newcommand{\cE}{{\mathcal E}}

\newcommand{\cK}{{\mathcal K}}

\newcommand{\cT}{{\mathcal T}}
\newcommand{\cU}{{\mathcal U}}

\newcommand{\diag}{{\mathrm{diag}}}

\newcommand{\Lone}{{{\mathbb L}_1}}

\newcommand{\argmin}{\mathop{\mathrm{argmin}}}

\newcommand{\bfone}{\mathbf 1}
\newcommand{\dom}{\mathrm{dom}}

\def\QED{\mbox{\rule[0pt]{1.3ex}{1.3ex}}} 

\newenvironment{proof}{{\quad \it Proof:\,}}{\hfill \QED \par}

\newtheorem{thm}{Theorem}
\newtheorem{cor}{Corollary}

\newtheorem{defn}{Definition}
\newtheorem{lem}{Lemma}

\newtheorem{prop}{Proposition}

\newenvironment{proof-of}[1]{{\quad\it Proof of #1:\,}}{\hfill\QED\par}
\title{An Optimal Control Formulation of Pulse-Based Control Using Koopman Operator
	\author{Aivar Sootla, Alexandre Mauroy and Damien Ernst
		\thanks{Aivar Sootla is with the Department of Engineering Science, University of Oxford, Parks Road, Oxford, OX1 3PJ, UK {\tt aivar.sootla@eng.ox.ac.uk}.}
		\thanks{Alexandre Mauroy is with Namur Center for Complex Systems (naXys) and Department of Mathematics, University of Namur, B-5000, Belgium {\tt alexandre.mauroy@unamur.be}}
		\thanks{Damien Ernst is with Montefiore Institute, University of Li\`{e}ge,  All\'ee de la D\'ecouverte 10, B-4000, Li\`{e}ge, Belgium. {\tt dernst@ulg.ac.be}}
		\thanks{This work was performed while A. Sootla held a postdoctoral position at University of Li\`{e}ge funded by F.R.S.--FNRS. Currently, Aivar Sootla is supported by the EPSRC Grant EP/M002454/1.}
	}
}
\begin{document}
\maketitle
\IEEEpeerreviewmaketitle

\begin{abstract}
	In many applications, and in systems/synthetic biology in particular, it is desirable to compute control policies that force the trajectory of a bistable system from one equilibrium (the initial point) to another equilibrium (the target point), or in other words to solve the switching problem. It was recently shown that, for monotone bistable systems, this problem admits easy-to-implement open-loop solutions in terms of temporal pulses (i.e., step functions of fixed length and fixed magnitude). In this paper, we develop this idea further and formulate a problem of convergence to an equilibrium from an arbitrary initial point. We show that this problem can be solved using a static optimization problem in the case of monotone systems.	Changing the initial point to an arbitrary state allows building closed-loop, event-based or open-loop policies for the switching/convergence problems. In our derivations, we exploit the Koopman operator, which offers a linear infinite-dimensional representation of an autonomous nonlinear system. One of the main advantages of using the Koopman operator is the powerful computational tools developed for this framework. Besides the presence of numerical solutions, the switching/convergence problem can also serve as a building block for solving more complicated control problems and can potentially be applied to non-monotone systems. We illustrate this argument on the problem of synchronizing cardiac cells by defibrillation. Potentially, our approach can be extended to problems with different parametrizations of control signals since the only fundamental limitation is the finite time application of the control signal.
\end{abstract}

\section{Introduction}

Synthetic biology is an active field of research with applications in metabolic engineering, bioremediation and energy sector~(\cite{Purnick:2009}). One of the main goals of synthetic biology is to engineer biological functions in living cells~(\cite{brophy2014principles}) and control theory naturally plays an essential role toward that end. Control theoretic regulation of protein levels in microbes was achieved by \cite{milias2011silico}, \cite{Menolascina:2011}, and \cite{uhlendorf2012long}. However, these works result in time-varying feedback control signals, which are affected by limitations due to physical constraints not only in sensing but also in actuation. Actuation limitations are observed with chemical induction, for instance. While the concentration of a chemical can be easily increased by adding this chemical to a culture, it is much more complicated to decrease its concentration (for instance, diluting is a labor-intensive process that is not suited to be performed repeatedly). Therefore, one of the challenges in synthetic biology is to derive control policies that not only achieve the desired objectives but are also simple enough to deal with the actuation limitations. In this context, an example of such control policies is based on temporal pulses, on which the present paper focuses.

One of the basic but nevertheless important control problems is \emph{the problem of convergence to an equilibrium}. In synthetic biology, it can, for instance, be motivated by the \emph{genetic toggle switch} circuit, which is a major building block in applications. The genetic toggle switch by \cite{Gardner00}, for example, consists of two interacting genes. Their design ensured that the concentration of proteins expressed by one gene is always much higher than the concentration of proteins expressed by the other gene: one gene is switched ``on'' while the other is switched ``off''. In this case, the system is bistable and the control objective of the problem is to drive the state from one equilibrium (e.g. one gene switched ``on'') to the other (e.g. the other gene switched ``on'') in minimum time and given a fixed energy budget (e.g. to avoid cell death). In this paper, we propose to solve this convergence problem using a specific set of signals, namely  
temporal pulses $u(t)$ with fixed length $\tau$ and magnitude $\mu$: 
\begin{equation}\label{eq:pulse}
u(t) = \mu h(t,\tau) \qquad 
h(t,\tau) = 
\begin{cases} 1 & 0 \leq t \leq \tau\,,\\
0 & t >\tau\,.
\end{cases}
\end{equation}
The main challenge in solving this convergence problem is the absence of well-developed theory and computational methods. This can be explained by the fact that classical optimal control methods rely on the semigroup property (of the value function and/or the dual variable), while our parametrization of control signals violates it. We, therefore, restrict our analysis to the class of monotone systems, properties of which allow simpler solutions. We note, however, that monotone systems play an important role in systems/synthetic biology and hence there exist many applications with this restriction. We solve the convergence problem by computing only one time-independent function, which we call the pulse control function and denote by $r$. This function links all the tunable parameters of the convergence problem so that finding a tradeoff between the convergence time and the energy budget becomes straightforward.

In our solution, we use the Koopman operator (cf.~\cite{mezic2005}), which offers a linear infinite dimensional description of a nonlinear system and allows a spectral approach to the nonlinear system. In particular, the definition of the pulse control function involves the dominant eigenfunction of the Koopman operator of the unforced system (i.e., when the control signal is equal to zero). This allows to use computational methods developed in the Koopman operator framework to solve our problem. In particular, we show that the function $r$ can be computed with Laplace averages (\cite{mauroy2013isostables}).

Our theoretical results generally do not apply to non-monotone systems. Let alone a solution using pulses may not exist. However, we can still use the function $r$ in some cases to derive control policies for non-monotone systems. Furthermore, we can use our framework to solve more complicated control problems. We illustrate this idea on the problem of  synchronization of cardiac cells modeled by non-monotone FitzHugh-Nagumo systems (\cite{fitzhugh1961impulses,nagumo1962active}).

\emph{Related Work.} \cite{sootla2015pulsesaut} proposed to solve the problem of switching between exponentially stable equilibria in monotone systems using temporal pulses. The authors derived a computational procedure to estimate the set of control signals allowing the switch. The control strategy proposed by~\cite{sootla2015pulsesaut} is open loop so that the control signal cannot be adjusted during the experiment. \cite{sootla2016nolcos} considered a similar setting, but using the Koopman operator framework to estimate the set of all switching pulses and provide estimates for convergence time. In this paper, we present an optimization program for solving the convergence problem, which is a generalization of the switching problem considered by~\cite{sootla2015pulsesaut} and~\cite{sootla2016nolcos}.
	
\cite{mauroy2014converging} considered the convergence/escape problem cast in the Koopman framework formalism. \cite{mauroy2014converging} did not restrict the class of systems but assumed full controllability, i.e. control signals affect all the states in an affine manner. \cite{Wilson2014} proposed to synchronize the cardiac cells by using geometric tools from Koopman operator framework and the techniques based on the Pontryagin's maximum principle. The authors did not parametrize the control signal, which led to complicated time-dependent optimal control signals. We will show that our scheme still achieves synchronization, but in the longer time than the optimal solution by~\cite{Wilson2014}.

To summarize, in this paper, we consider a convergence problem similarly to~\cite{mauroy2014converging}, the target set is chosen similarly to~\cite{Wilson2014}, while the class of systems and the control signals are restricted as in the works by~\cite{sootla2015pulsesaut} and~\cite{sootla2016nolcos}. We opted for restricting the class of systems in order to show optimality of our approach in this specific case, which indicates that systems' properties can be used to derive easy-to-implement and interpretable optimal solutions. 

\emph{Outline of the Paper}. The rest of the paper is organized as follows. In Section~\ref{s:prel}, we cover basic definitions and properties of monotone systems and we introduce the Koopman operator framework. We formulate and discuss our convergence problem in Section~\ref{s:sw-con-prob}, while also presenting the main result, which we prove in Section~\ref{s:proof}. In Section~\ref{s:examples}, we apply the theoretical results to the switching problem (open and closed loop control) and to the synchronization of cardiac cells.

\section{Preliminaries}
\label{s:prel}

Consider a system of the form
\begin{equation}
\label{sys:f}
\dot x = f(x,u),\quad x(0) = x_0,
\end{equation} 
with $f: \cD\times \cU\rightarrow \R^n$, $u:\R \rightarrow \cU$, and where $\cD\subset\R^n$, $\cU\subset\R$ are open and $u$ belongs to the space $\cU_{\infty}$ of Lebesgue measurable functions with values from $\cU$. We assume that $f(x, u)$ is twice continuously differentiable ($C^2$) in $(x, u)$ on $\cD\times \cU$. 
The flow map $\phi: \R \times \cD \times \cU_{\infty}\rightarrow \R^n$ induced by the system is such that $\phi(t, x_0, u)$ is a solution of~\eqref{sys:f} with an initial condition $x_0$ and a control signal $u$. We denote the differential of a function $g(x, y):\R^n\times\R^m\to \R^k$ with respect to $x$ as $\partial_x g(x, y)$.  Let $J(x)$ denote the Jacobian matrix of $f(x,0)$ (i.e., $J(x) = \partial_x f(x,0)/\partial x$).  
For every stable fixed point $x^\ast$ of the system, we assume that the eigenvectors of $J(x^\ast)$ are linearly independent (i.e., $J(x^\ast)$ is diagonalizable). The eigenvalues of $J(x^\ast)$ are denoted by $\lambda_i$ with $i = 1,\dots, n$ and are ordered by their real part, that is $\Re(\lambda_i) \ge \Re(\lambda_j)$ for all $i$, $j$.  
We also denote the positive orthant by $\Rp^n = \{x\in\R^n | x_i > 0,\, i = 1,\dots, n\}$, and the nonnegative orthant by $\Rnn^n = \{x\in\R^n | x_i \ge 0,\, i = 1,\dots, n\}$.

\subsection{Koopman operator}

Autonomous nonlinear systems can be studied in the framework of the Koopman operator. A semigroup of Koopman operators acts on functions $g:\R^n\rightarrow \C$ (also called observables) and is defined by
\begin{gather}
U^t g(x) = g \circ \phi(t,x,0), \qquad t \geq 0
\end{gather}
where $\circ$ is the composition of functions. Provided that the vector field and observables are $C^1$, one can define the infinitesimal generator of the operator as $L g(x) = (f(x, 0))^T \nabla g(x)$ on a compact set. The semigroup is linear (cf.~\cite{mezic2013analysis}) and can be studied through its spectral properties. In this paper, we will limit our use of the Koopman operator to unforced systems~\eqref{sys:f} on a basin of attraction of an exponentially stable equilibrium $x^\ast$ (that is, the eigenvalues $\lambda_j$ of $J(x^\ast)$ are such that $\Re(\lambda_j) < 0$ for all $j$). The basin of attraction is defined by $\cB(x^\ast) = \{x\in \R^n | \lim_{t\rightarrow\infty}\phi(t,x,0) = x^\ast \}$. In this case, the eigenvalues $\lambda_j$ of the Jacobian matrix $J(x^\ast)$ are also the so-called Koopman eigenvalues, which are associated with the Koopman eigenfunctions $s_j: \cB \rightarrow \C$ satisfying
\begin{gather}
\label{eq:property_eigenf}
U^t s_j(x) = s_j(\phi(t, x,0)) = s_j(x) \, e^{\lambda_j t}, \quad x\in \cB,
\end{gather}
or equivalently
\begin{gather}
\label{eq:property_eigenf2}
\nabla s_j(x)^T f(x, 0) = \lambda_j s_j(x).
\end{gather}
If the vector field $f(\cdot, 0)$ is a $C^2$ function and the Jacobian matrix $J(x^\ast)$ is diagonalizable, then the Koopman eigenfunctions $s_j$ belong to $C^1$~(\cite{mauroy2014global}).

We refer to an eigenvalue $\lambda_1$ satisfying $\Re(\lambda_1) > \Re(\lambda_j)$ for all $\lambda_1 \ne \lambda_j$ as \emph{the dominant eigenvalue}. We assume that such an eigenvalue exists (it is the case for monotone systems) and we call the associated eigenfunction $s_1$ \emph{the dominant eigenfunction}. If the dominant eigenvalue is such that $\Re(\lambda_1)<0$, then the dominant eigenfunction $s_1$ can be computed through the Laplace average
\begin{gather}\label{laplace-average}
g_\lambda^\ast(x) = \lim\limits_{t\rightarrow \infty}\frac{1}{T}\int\limits_0^T (g\circ \phi(t, x,0)) e^{-\lambda t} d t.
\end{gather}
For all $g\in C^1$ that satisfy $g(x^\ast)=0$ and $(\nabla g(x^\ast))^T v_1 \neq 0$, where $v_1$ is the right eigenvector of $J(x^\ast)$ corresponding to $\lambda_1$, the Laplace average $g_{\lambda_1}^\ast$ is equal to $s_1(x)$ up to a multiplication with a scalar. If the algebraic and geometric multiplicities of $\lambda_1$ are equal to $\mu_1$, then there are $\mu_1$ independent eigenfunctions associated with $\lambda_1$ and they can be computed by choosing linearly independent right eigenvectors of $J(x^\ast)$ corresponding to $\lambda_1$. The eigenfunctions $s_j(x)$ with $j\ge 2$ are generally harder to compute using Laplace averages, but can be obtained through other methods such as linear algebraic methods by~\cite{mauroy2014global}. The eigenfunctions can also be estimated from data by using the so-called dynamic mode decomposition (DMD) method (cf.~\cite{Schmid2010,Tu2014}). 

The Koopman eigenfunctions capture important geometric properties of the system. In particular, the dominant Koopman eigenfunction $s_1$ is related to the notion of \emph{isostables}. 

\begin{defn} Suppose that $s_1 \in C^1$. An isostable $\partial \cB_\alpha$ associated with the value $\alpha > 0$ is the boundary of the set
$\cB_\alpha =  \{x\in \R^n \,|\, |s_1(x)| \le \alpha \}$, that is
\begin{equation*}
\partial \cB_\alpha =  \{x \in \R^n \,|\, |s_1(x)| = \alpha \}.
\end{equation*}
\end{defn}

A more rigorous definition based on the flow is due to~\cite{mauroy2013isostables}. Isostables are useful from a control perspective since they capture the dominant (or asymptotic) behavior of the unforced system. Indeed, it can be shown that trajectories with initial conditions on the same isostable $\partial \cB_{\alpha_1}$ converge synchronously toward the equilibrium, and reach other isostables $\partial \cB_{\alpha_2}$, with $\alpha_2<\alpha_1$, after a time
\begin{equation}
\label{time_isostable}
\cT = \frac{1}{|\Re(\lambda_1)|} \ln \left(\frac{\alpha_1}{\alpha_2}\right)\,.
\end{equation}
In the case $\lambda_1 \in \R$, for example, it can be shown that the trajectories starting from $\partial \cB_\alpha$ share the same asymptotic evolution
\begin{equation*}
\phi(t,x,0) \rightarrow x^\ast + v_1 \, \alpha  e^{\lambda_1 t}\,, \quad t\rightarrow \infty\,.
\end{equation*}

\subsection{Monotone Systems}
We consider systems that preserve a partial order induced by a nonnegative orthant $\Rnn^n$, but our results can be extended to other cones $\cK$ in $\R^n$. 
We define \emph{a partial order} $\succeq$ as follows: $x\succeq y$ if and only if $ x - y \in \Rnn^n$ (we write $x\not \succeq y$ if the relation  $x \succeq y$ does not hold). We will also write $x\succ y$ if $x\succeq y$ and $x\ne y$, and $x\gg y$ if $x- y \in \Rp^n$.  Similarly, a partial order can be defined on the space of signals $u\in \cU_{\infty}$: $u\succeq v$ if $u(t) - v(t) \in \Rnn^n$ for all $t\ge 0$. 
We also introduce concepts that are important for our subsequent discussion. Let $[x,~y]$ denote an interval in the order $\preceq$, that is $[x,~y] = \{ z \in\R^n | x\preceq z \preceq y \}$. 	
For a function $W:\R^n \rightarrow \R\cup \{-\infty, +\infty\}$, we refer to  the set $\dom(W) = \{x\in \R^n | |W(x)| < \infty  \}$ as its effective domain. A function $W:\R^n \rightarrow \R\cup \{-\infty, +\infty\}$ is called \emph{increasing} if $W(x) \ge W(y)$ for all $x\succeq y$ on $\dom(W)$. Control systems in the form~\eqref{sys:f} whose flows preserve a partial order relation $\succeq$ are called \emph{monotone systems}. 
\begin{defn}\label{def:mon}
The system $\dot x = f(x,u)$ is called \emph{monotone} if $\phi(t,x, u)\preceq \phi(t,y, v)$ for all $t\ge 0$, and for all $x\preceq y$, $u\preceq v$. 
\end{defn}
\begin{defn} The unforced system $\dot x = f(x, 0)$ is \emph{strongly monotone} 
if it is monotone and $x\prec y$ implies that $\phi(t,x, 0)\ll \phi(t,y, 0)$ for all $t> 0$.
\end{defn}
A certificate for monotonicity is a condition on the vector field, for which we refer the reader to~\cite{angeli2003monotone}. We finally consider the spectral properties of unforced monotone systems that are summarized in the following result. 
\begin{prop} \label{prop:mon-eig-fun}
Consider that the system $\dot x = f(x)$ with $f\in C^2(\cD)$ has an exponentially stable equilibrium $x^\ast$ and let $\lambda_j$ be the eigenvalues of $J(x^\ast)$ such that $\Re(\lambda_i)  \ge \Re(\lambda_j)$ for all $i \le j$.\\
(i)  If the system is monotone with respect to $\cK$ on a set $\cC\subseteq \cB(x^\ast)$, then $\lambda_1$ is real and negative, the right eigenvector $v_1$ of $J(x^\ast)$ can be chosen such that $v_1 \succ 0$, while the eigenfunction $s_1$ can be chosen such that $s_1(x) \ge s_1(y)$ for all $x$, $y\in \cC$ satisfying $x\succeq y$. \\
(ii) Furthermore, if the system is strongly monotone with respect to $\cK$ on a set $\cC \subseteq \cB(x^\ast)$ then $\lambda_1$ is simple, real and negative, $\lambda_1 > \Re(\lambda_j)$ for all $j \ge 2$, $v_1$ and $s_1$ can be chosen such that $v_1\gg 0$ and $s_1(x)> s_1(y)$ for all $x$, $y\in \cC$ satisfying $x\succ y$; 
\end{prop}
The result is from~\cite{sootla2016basins}. Without loss of generality, we will assume that a dominant eigenfunction $s_1$ is increasing even if $\lambda_1$ is not simple. 

\section{Convergence to an Isostable Problem} \label{s:sw-con-prob}
\subsection{Problem Formulation and Discussion}
In order to formulate the basic problem we want to address, consider the following assumptions:
\begin{enumerate}
	\item[{\bf A1.}] The vector field $f(x,u)$ in~\eqref{sys:f} is twice continuously differentiable in $(x, u)$ on $\cD\times \cU$. 
	\item[{\bf A2.}] The unforced system~\eqref{sys:f} has an exponentially stable equilibrium $x^\ast$ in $\cD$ with a diagonalizable $J(x^\ast)$. 
	\item[{\bf A3.}] The system is monotone with respect to $\Rnn^n\times\R$ and forward-invariant on $\cD\times\cU$ in the sense that for all $x\in \cD$, $u\in\cU_{\infty}$, the flow $\phi(t,x,u(\cdot))$ belongs to $\cD$. 
	\item[{\bf A4.}] The eigenfunction $s_1(x)$ is such that 
$\nabla s_1(x) \gg 0$ for all $x\in \dom(s_1)$. 
	\item[{\bf A5.}] $f(x, \mu_1) \succ f(x, \mu_2)$ for all $x\in\cD$ and $\mu_1> \mu_2\ge 0$.
	\item[{\bf A6.}] The space of control signals is limited to temporal pulses $u(t) = \mu h(t,\tau)$, where $h$ is defined in \eqref{eq:pulse}.
	\end{enumerate}
Assumption~{A1} guarantees existence and uniqueness of solutions, while Assumption~{A2} introduces a reference point $x^\ast$. These assumptions are perhaps more restrictive than the ones usually met in control theory. That is, $f(x,u)$ is usually assumed to be Lipschitz continuous in $x$ for every fixed $u$, and the equilibria are asymptotically stable. Our assumptions are guided by our consequent use of the Koopman operator. Assumptions~A1 and~A2 guarantee the existence of continuously-differentiable eigenfunctions on the basin of attraction $\cB(x^\ast)$ of $x^\ast$. Monotonicity is crucial, but forward-invariance on $\cD\times\cU$ is a rather technical assumption on which our computational methods do not rely. Assumption~A4 is well-posed since $s_1\in C^1$ due to Assumption~A1. If Assumptions A1 -- A3 hold, then we have $\nabla s_1(x) \succ 0$ and $f(x,\mu_1) \succeq f(x, \mu_2)$ for $\mu_1 > \mu_2 \ge0$, hence Assumptions~A4 and~A5 serve as technical assumptions that guarantee uniqueness of solutions and a certain degree of regularity. We will comment throughout the paper on the case when Assumptions A4 -- A5 do not hold. Assumption~A6 is guided by many applications, where there is a need to parametrize in advance the control signal. In this paper, we choose the easiest parametrization, although the only fundamental limitation is to set $u(t)$ to zero after some time $\tau$. It is, therefore, possible to generalize our approach to more complicated control signals. We proceed by formulating a basic but fundamental problem.

\emph{Problem 1. Converging to an isostable.} Consider the system $\dot x = f(x,u)$ satisfying assumptions {A1--A6} and the initial state $x_0$. Compute a control signal $u(t) = \mu h(t,\tau)$ such that the flow $\phi(t,x_0, u(\cdot))$ reaches the set $\cB_\varepsilon(x^\ast)$ for some small $\varepsilon>0$ in minimum time units $\cT_{\rm conv}$ subject to the energy budget $\|u\|_\Lone \le \cE_{\rm max}$.

Our formulation based on the isostables is guided by our use of Koopman operator framework for computational purposes. However, there are other benefits in this formulation. One can view $\cB_\varepsilon(x^\ast)$ as a ball in the (contracting) pseudometric $d_K(x,y) = |s_1(x) - s_1(y)|$ on a basin of attraction $\cB(x^\ast)$  ($d_K(x, y)$ is a pseudometric, since $d_K(x, y)$ can be equal to zero for some $x\ne y$). By reformulating the standard convergence problem using a pseudometric defined through Koopman eigenfunctions, we take into account the dynamical properties of the unforced system. Furthermore, if $\varepsilon$ is close to zero then the solution of Problem 1 can be used to solve a convergence-type problem. For example, the problem of switching between equilibria, which is considered in~\cite{sootla2016nolcos, sootla2015pulsesaut}, falls into this category. 

The main challenge in solving this problem is the parametrization of the control signal. Most of the control methods (such as dynamic programming, Pontryagin's maximum principle) are not tailored to deal with time parametrized control signals since they rely on the semigroup property of the value function or the dual variable. Hence it is not entirely clear how to systematically approach this problem through these methods.

\subsection{A Solution using a Static Optimization Program}
The computational solution to our problem will be established by computing first a static function, which we call \emph{the pulse control function} and define below. 

\begin{defn} \label{def:control-smth-function} Let the function $r:\cD \times \Rnn\times \Rnn \to \C\bigcup\{\infty\}$ such that 
	\begin{gather*}
	r(x, \mu,\tau) = s_1(\phi(\tau, x, \mu)),
	\end{gather*}
	where $s_1$ is a dominant eigenfunction on the basin of attraction of $x^\ast$, be called \emph{the pulse control function}. By convention $r(x, \mu, \tau) = \infty$, if $\phi(\tau, x, \mu) \not \in \cB(x^\ast)$
\end{defn}

If $s_1$ is real-valued and increasing on $\dom(s_1) = \cB(x^\ast)$, then we assume it is extended to $\R^n$ so that $s_1:\R^n\to \R\cup \{-\infty, +\infty\}$ is increasing on $\R^n$. We note that the switching function proposed in~\cite{sootla2016nolcos} corresponds to $r(x^\ast, \cdot,\cdot)$. The pulse control function $r$ can be used in a context broader than the switching function by~\cite{sootla2016nolcos}, which we demonstrate in this paper. In particular, we will solve Problem~1 using the following result, which we will prove in what follows.

\begin{thm}
\label{thm:main_result}
	Consider the system~\eqref{sys:f}, \emph{Problem 1} under Assumptions~{A1--A6} and the optimization program:
\begin{align}
\label{prog:static}\gamma^\ast= \min\limits_{\mu\ge 0, \tau\ge 0}~~ 
&\frac{1}{|\lambda_1|} ~~\ln\left(|r(x, \mu,\tau)|\right) + \tau, \\
\label{con:time}\text{subject to:}~~&  r(x,\mu,\tau)\le -\varepsilon, \\
\label{con:energy}                                 & \mu\cdot \tau \le \cE_{\rm max}.
\end{align}	
	If $s_1(x_0)\le-\varepsilon$, an optimal solution to~\eqref{prog:static} is an optimal solution to \emph{Problem 1}, if the former is feasible. Furthermore, the objective is nonincreasing in $\mu$ and $\tau$ and an optimal solution to~\eqref{prog:static}, if it exists, is achieved at the boundary of the admissible set to the constraint~\eqref{con:time} and/or the constraint~\eqref{con:energy}.
\end{thm}

If Assumptions~A5 and~A6 do not hold, then we can possibly have multiple minima including the points which do not activate the constraints. However, we can still compute a minimizing solution using the same program. Intuitively, the objective function is the convergence time (in fact, $\cT_{\rm conv}=\gamma^\ast-1/|\lambda_1| \ln(\varepsilon)$), the constraint on $r(x, \mu, \tau) \le - \varepsilon$ ensures that we stop when reaching $\cB_\varepsilon$ and the constraint $\mu\cdot \tau \le \cE_{\rm max}$ is the energy budget. Since the optimum is attained when one of the constraints is active, the optimization program can be solved by a line search over $\mu$ (or $\tau$) over the constraints curves, provided that $r$ can be estimated at any given point. In our simulations we compute the function $r$ for specific pairs of values $(\mu,\tau)$ and take the minimum over these pairs. The two terms in the objective function of the static optimization problem show the tradeoff on the choice of the intermediate target isostable (which is to be reached after a time $\tau$). For instance, choosing an isostable close to the equilibrium can lead to a large pulse duration $\tau$ (second term), but a small convergence time of the free motion (first term). Furthermore, the function $r$ also allows to understand the tradeoff between the energy spent and the convergence time, which is not straightforward using standard optimal control theory. To summarize, we derived a \emph{static optimization problem}, which has the same solution as the \emph{dynamic optimization problem} (Problem 1). In order to compute the solution, one needs an efficient computational procedure for evaluating $r$ at a given point. 
\subsection{Computation of the Pulse Control Function}
The eigenfunction $s_1$ can be estimated through Laplace averages~\eqref{laplace-average} and the function $r$ is subsequently obtained since it is the composition of the eigenfunction $s_1$ with the flow. In particular, we can derive the following formula:
\begin{multline}
r(x,\mu,\tau)  = \lim\limits_{\bar{t}\rightarrow\infty} \frac{1}{\bar{t}} \int\limits_0^{\bar{t}} g\circ \phi(t, \phi(\tau,x, \mu),0) e^{-\lambda_1 t}d t \\
= \lim \limits_{\bar{t}\rightarrow\infty} \frac{1}{\bar{t}}\int\limits_\tau^{\bar{t}} g\circ \phi(t, x, \mu h(\cdot,\tau)) e^{-\lambda_1 (t-\tau)}d t ,
\label{eigen-fun-cont}
\end{multline}
where $\lambda_1$ is the dominant Koopman eigenvalue, $g\in C^1$ satisfies $g(x^\ast)=0$, $v_1^T \nabla g(x^\ast) \neq 0$, $v_1$ is the right eigenvector of $J(x^\ast)$ corresponding to $\lambda_1$ and $h(t,\tau)$ is the step function defined in~\eqref{eq:pulse}. In practice, we choose $g(x)= w_1^T (x- x^\ast)$, where $w_1$ is the dominant left eigenvector of $J(x^\ast)$. Since $\lambda_1$ is real according to Assumption A3 and Proposition~\ref{prop:mon-eig-fun}), we have 
\begin{gather*}
\begin{split}
r(x, \mu, \tau)  & = \lim_{\bar{t} \rightarrow \infty} w_1^T(\phi(\bar{t}, x, \mu h(\cdot,\tau)) - x^\ast) e^{-\lambda_1 (\bar{t}-\tau)} \\
& \approx w_1^T(\phi(\bar{t}, x, \mu h(\cdot,\tau)) - x^\ast) e^{-\lambda_1 (\bar{t}-\tau)}
\end{split}
\end{gather*}
where the time $\bar{t}$ should be chosen large enough.  In this case the tolerance of the differential equation solver should be set to $O(e^{\lambda_1 (\bar{t}-\tau)})$. When only observed data are available, the eigenfunction --- and therefore the function $r$ --- can be computed through dynamic mode decomposition methods (cf.~\cite{Schmid2010,Tu2014}). This idea is illustrated in Appendix~\ref{app:DMD}.

\subsection{Is the Control Space Rich Enough?} \label{s:rich_control}
Throughout the paper, we assume that the problem has a solution in the form of a temporal pulse $u(t) = \mu h(t, \tau)$. We will argue that in the case of monotone systems, this is not a restrictive assumption. First of all, if the system is globally asymptotically stable, then clearly we can converge to $x^\ast$ by using a temporal pulse.

Assume now that the system is monotone with two exponentially stable equilibria $x^\ast$ and $x^\bullet$ and basins of attraction $\cB(x^\ast)$ and $\cB(x^\bullet)$, respectively. Let the system be defined on a forward-invariant set $\cD = \cB(x^\ast)\bigcup\cB(x^\bullet)$. Assume also that $x^\ast \gg x^\bullet$, which is typically fulfilled in many bistable monotone systems. If $x_0\in \cB(x^\ast)$, then we can choose $u = 0$, which is a temporal pulse with $\tau =0$. Consider now the case $x_0 = x^\bullet$. If there exists a control signal $u^1\in \cU_\infty$ driving the system from $x^\bullet$ to $x^\ast$, then we have $\phi(t, x^\bullet, u^1) \preceq \phi(t, x^\bullet, \mu)$, where $u^1(t)\le \mu$ for (almost) all $t$. At a time $\tau$, the flow $\phi(\tau, x^\bullet, u^1)$ will be in the vicinity of $x^\ast$ and in the basin of attraction of $\cB(x^\ast)$. The flow $\phi(\tau, x^\bullet, \mu)$ will also be in the basin of attraction of $x^\ast$. Indeed, if $\phi(\tau, x^\bullet, \mu) \in \cB(x^\bullet)$, then $[x^\bullet, \phi(\tau, x^\bullet, \mu)]\in \cB(x^\bullet)$, which contradicts that $\phi(\tau, x^\bullet, u^1) \in \cB(x^\ast)$ and $\phi(\tau, x^\bullet, u^1) \in[x^\bullet, \phi(\tau, x^\bullet, \mu)]$ (cf.~\cite{sootla2016basins}). Hence if we can switch from $x^\bullet$ to $x^\ast$ with a control signal $u(t)$, then we can switch with a temporal pulse. Finally the case $x_0 \in \cB(x^\bullet)$ is treated in a similar manner by first allowing the trajectory to converge to a neighborhood $x^\bullet$ with $u^2 =0$ and then applying the argument above.

This discussion shows that using temporal pulses in the case of monotone systems does not restrict the space of feasible problems. However, we can strengthen the argument by showing that \emph{constant controls are optimal} in the absence of energy constraints in Appendix~\ref{app:const-control}.

\section{Proof of the Main Result} \label{s:proof}
In order to prove the main result, we would need to use the properties of the pulse control function $r$, which we present in the following lemma.

\begin{lem}\label{lem:r-prop}
Let the system~\eqref{sys:f} satisfy Assumptions~{A1--A5}. Then $r$ is a $C^1$ function on its effective domain $\dom(r)$. Furthermore, for all $(x,\mu,\tau)\in\dom(r)$\\
(i)  $\partial_x    r(x,\mu,\tau) \gg 0$,   $\partial_\mu r(x,\mu,\tau) > 0$  and  $\partial_\tau r(x,\mu,\tau) > \lambda_1 r(x,\mu,\tau)$\\
(ii) If $r(x,\mu,\tau)\le 0$, then $\partial_\tau r(x,\mu,\tau)> 0$;\\
(iii) If $f(x, \nu)\succeq 0$, then $\partial_\tau r(x,\mu,\tau)>0$ for all finite $\tau> 0$ and $\mu > \nu$.
\end{lem} 
\begin{proof}
(o) First, we show that under the assumptions above for all  $t >0$ and $\mu_1 > \mu_2$, we have
\begin{gather}
\label{eq1:a6-prop}	\phi(t, x, \mu_1) \succ 	\phi(t, x, \mu_2), \\
\label{eq2:a6-prop}	s_1(\phi(t, x, \mu_1)) > s_1(\phi(t, x, \mu_2)).
\end{gather}
Due to monotonicity, we have that $\phi(t,x, \mu_1) \succeq \phi(t,x, \mu_2)$ for all  $t >0$  and  $\mu_1 > \mu_2$. All we need to show is that $\phi(t,x, \mu_1) \ne \phi(t,x, \mu_2)$ for all finite $t>0$. At $t=0$, the time derivatives of the flow are equal to $f(x,\mu_1)$ and $f(x,\mu_2)$. Since $f(x,\mu_1) \succ f(x,\mu_2)$ (Assumption~{A5}), there exists a $T>0$ such that $\phi(t,x, \mu_1) \succ \phi(t,x, \mu_2)$ for all $t<T$. If for some $T$ we have that $\phi(T,x, \mu_1) = \phi(T,x, \mu_2)$ and $\phi(t,x, \mu_1) \succ \phi(t, x, \mu_2)$ for all $t <T$, then for some index $i$ we have
\[
\frac{d \phi_i(t,x, \mu_1)}{d t}\Bigl|_{t =T} <  \frac{d \phi_i(t,x, \mu_2)}{d t}\Bigl|_{t =T}.
\]
This implies that $f_i(\phi(T, x, \mu_1),\mu_1) < f_i(\phi(T, x, \mu_2), \mu_2)$, which together with $\phi(T, x, \mu_1) = \phi(T,x, \mu_2)$ contradicts Assumption~A5. Therefore, $\phi(t,x, \mu_1) \succ \phi(t,x, \mu_2)$ for all finite $t> 0$. Due to Assumption~{A4} we have that $\nabla s_1(x )\gg 0$, which in particular means that $s_1(x) > s_1(y)$ for all $x\succ y$, and~\eqref{eq2:a6-prop} follows from~\eqref{eq1:a6-prop}.

(i)  The flow is continuously-differentiable for constant control signals since $f(x, u)\in C^2$ (Assumption~A1), and hence $r(x,\mu,\tau) = s_1(\phi(\tau,x,\mu))$ is a $C^1$ function. \\
For $x \succ y$, where  $(x,\mu,\tau)$ and $(y,\mu,\tau)\in\dom(r)$, we have $s_1(\phi(\tau, x,\mu)) > s_1(\phi(\tau, y, \mu))$ due to monotonicity (Assumption~{A3}) and Assumption~A4. Hence $\nabla_x r(x, \mu,\tau)\gg 0$.\\
For $\mu > \nu$, where  $(x,\mu,\tau)$ and $(x, \nu,\tau)\in\dom(r)$, we have $s_1(\phi(\tau, x, \mu)) > s_1(\phi(\tau, x, \nu))$ due to monotonicity (Assumption~{A3}) and point (o). Hence $\partial_\mu r(x, \mu,\tau)> 0$. \\
Finally,  $\partial_\tau r(x, \mu,\tau)> \lambda_1 r(x,\mu, \tau)$ follows from:
\begin{multline*} 
\partial_\tau r(x,\mu, \tau)=\dfrac{d s_1(\phi(t,x,\mu))}{dt}\Bigl|_{t =\tau} = \\
\nabla s_1(\phi( \tau, x,\mu))^T f(\phi(\tau, x,\mu), \mu) > \\
\nabla s_1(\phi( \tau, x,\mu))^T f(\phi( \tau, x,\mu), 0) = \\
\lambda_1 s_1(\phi(\tau, x, \mu))=\lambda_1 r(x,\mu, \tau),
\end{multline*}
where the inequality is due to Assumption~A4 and A5, and the following equality is due to \eqref{eq:property_eigenf2}. 

(ii) This follows directly from point (i). 

(iii) This proof employs a fairly standard technique in monotone system theory. First note that $f(x,\nu)\succeq 0$ and Assumption~A5 imply that $f(x,\mu)\succ 0$. Consider a perturbed system $\dot z = f(z,\mu) +1/n \bfone$, where $\bfone$ is a vector of ones and $n$ is a positive integer. Let the flow of this system be $\phi_n(t,x,\mu)$. We have that $f(x,\mu) +1/n \bfone\gg f(x,\mu)\succ 0$. 	Since this is the derivative of the flow with respect to time around $t=0$, then $\phi_n(\delta, x, \mu)\gg x$ for a sufficiently small positive $\delta$. Now for $t > \delta$, we have
\begin{gather*}
\phi_n(t, x, \mu)  = \phi_n(t-\delta, \phi_n(\delta, x, \mu), \mu) \gg \phi_n(t-\delta, x, \mu),
\end{gather*}		
and hence $\phi_n(t, x, \mu)  \gg \phi_n(\xi, x, \mu)$	for any $t>\xi$. This conclusion holds for all $n>0$. With $n\rightarrow\infty$, we have that $\phi_n(t,x,\mu)\rightarrow \phi(t,x,\mu)$. Therefore, $\phi(t,x, \mu) \succeq \phi(\xi, x, \mu)$ for all finite $t>\xi$ and $\partial_\tau \phi(\tau, x, \mu) \succeq 0$.  If the equality $\partial_\tau \phi(\tau, x, \mu) = 0$ is attained, then $\phi(\tau, x, \mu)$ is an equilibrium of the system $\dot x = f(x, \mu)$. This is impossible since an exponentially stable equilibrium cannot be reached in finite time $\tau$ due to uniqueness of solutions. Hence we have $\phi(t,x, \mu) \succ \phi(\xi, x, \mu)$ for all finite $t>\xi$ and, using Assumption~A4, we obtain $s_1(\phi(t,x,\mu)) > s_1(\phi(\xi, x, \mu))$ for all finite $t>\xi$. The result follows.
\end{proof}

If Assumptions~A4 and~A5 do not hold, then all the inequalities in Lemma~\ref{lem:r-prop} are not strict. For instance, we have that  $\partial_x r(x,\mu,\tau) \succ 0$, $\partial_\mu r(x,\mu,\tau) \ge 0$  and $\partial_\tau r(x,\mu,\tau) \ge \lambda_1 r(x,\mu,\tau)$ in point (i). We present additional properties of the function $r$ in Appendix~\ref{app:additional-res}. Now we can present the proof of our main result.

\begin{proof-of}{Theorem~\ref{thm:main_result}}
It is straightforward to verify that all the constraints and optimization objective are the same for \emph{Problem~1} and problem~\eqref{prog:static}. Hence, by construction the first part of the statement is fulfilled.\\
According to the constraint~\eqref{con:time}, we have that $r(x_0,\mu,\tau)<0$, which implies the following chain of inequalities
\begin{multline*}
\partial_\tau (\ln(|r(x_0,\mu,\tau)| e^{|\lambda_1| \tau})) = 
\frac{\partial_\tau (|r(x_0,\mu,\tau)| e^{|\lambda_1| \tau})}{|r(x_0,\mu,\tau)| e^{|\lambda_1| \tau}} = \\
\frac{-\partial_\tau (r(x_0,\mu,\tau)) \cdot  e^{|\lambda_1| \tau} +  |\lambda_1|  |r(x_0,\mu,\tau)| e^{|\lambda_1| \tau} }{|r(x_0,\mu,\tau)| e^{|\lambda_1| \tau}} < \\\frac{\lambda_1  |r(x_0,\mu,\tau)| e^{|\lambda_1| \tau} +  |\lambda_1|  |r(x_0,\mu,\tau)| e^{|\lambda_1| \tau}}{|r(x_0,\mu,\tau)| e^{|\lambda_1| \tau}} = 0
\end{multline*}
where the inequality follows from Lemma~\ref{lem:r-prop}. Hence, the derivative of the objective function in~\eqref{prog:static} with respect to $\tau$ is negative. Finally, $\partial_\mu \ln(|r(x_0,\mu,\tau)| e^{|\lambda_1| \tau})$ is also negative according to Lemma~\ref{lem:r-prop}. Hence, if there is a feasible point, the constraints~\eqref{con:time}, \eqref{con:energy} are reached in order to minimize the objective. 
\end{proof-of}

\begin{figure*}[h]
	\centering
	\includegraphics[height = 0.5\columnwidth]{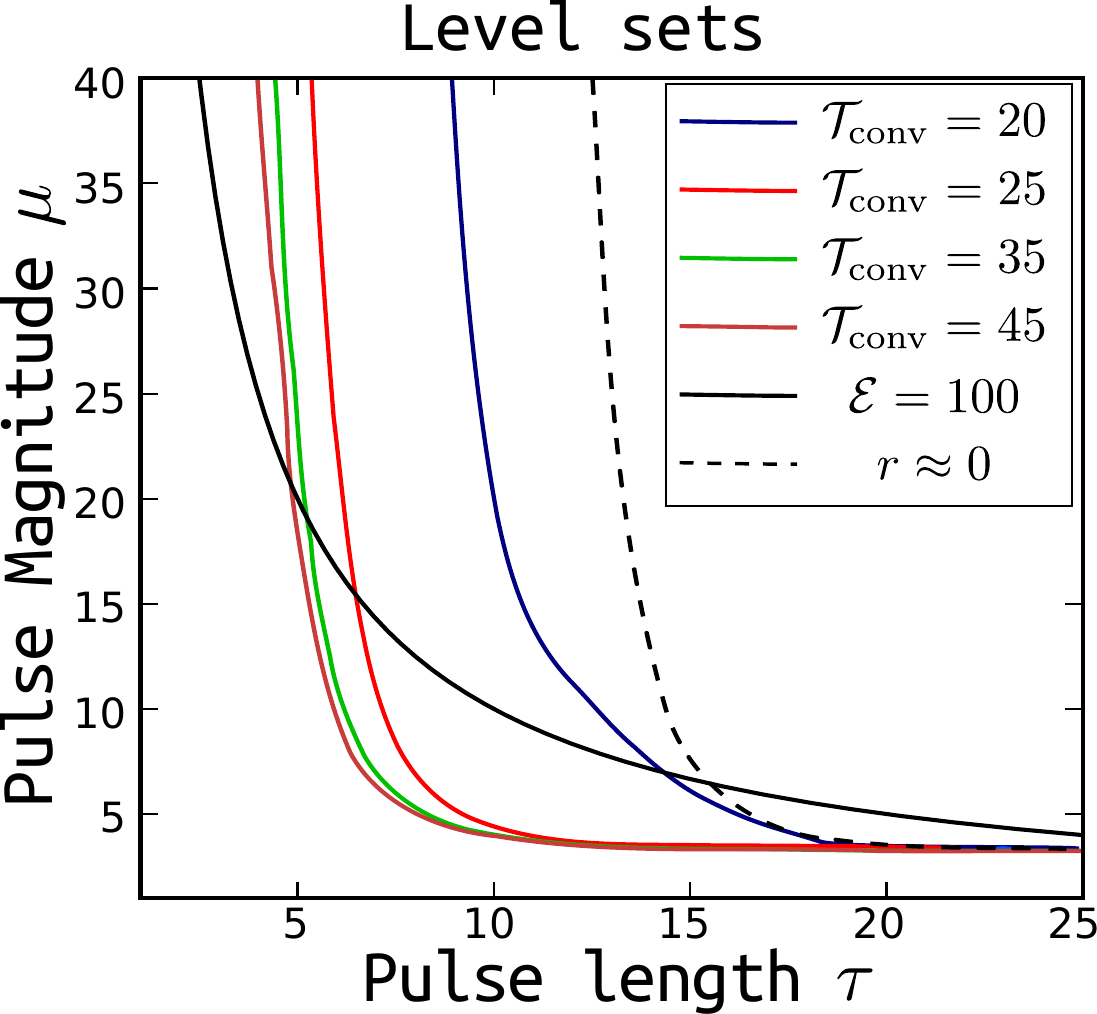}\hskip 10pt\includegraphics[height = 0.5\columnwidth]{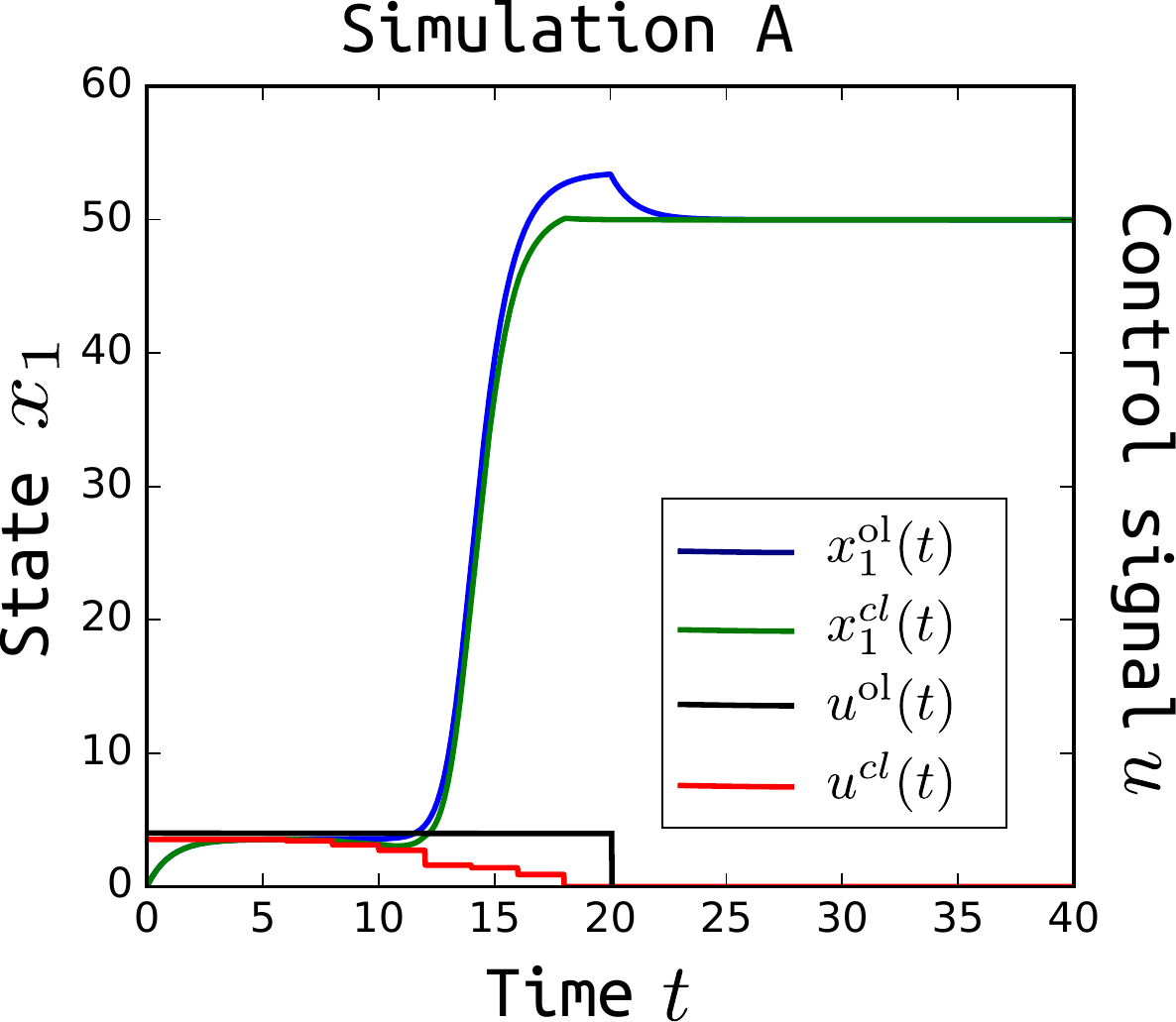}\hskip 10pt \includegraphics[height = 0.5\columnwidth]{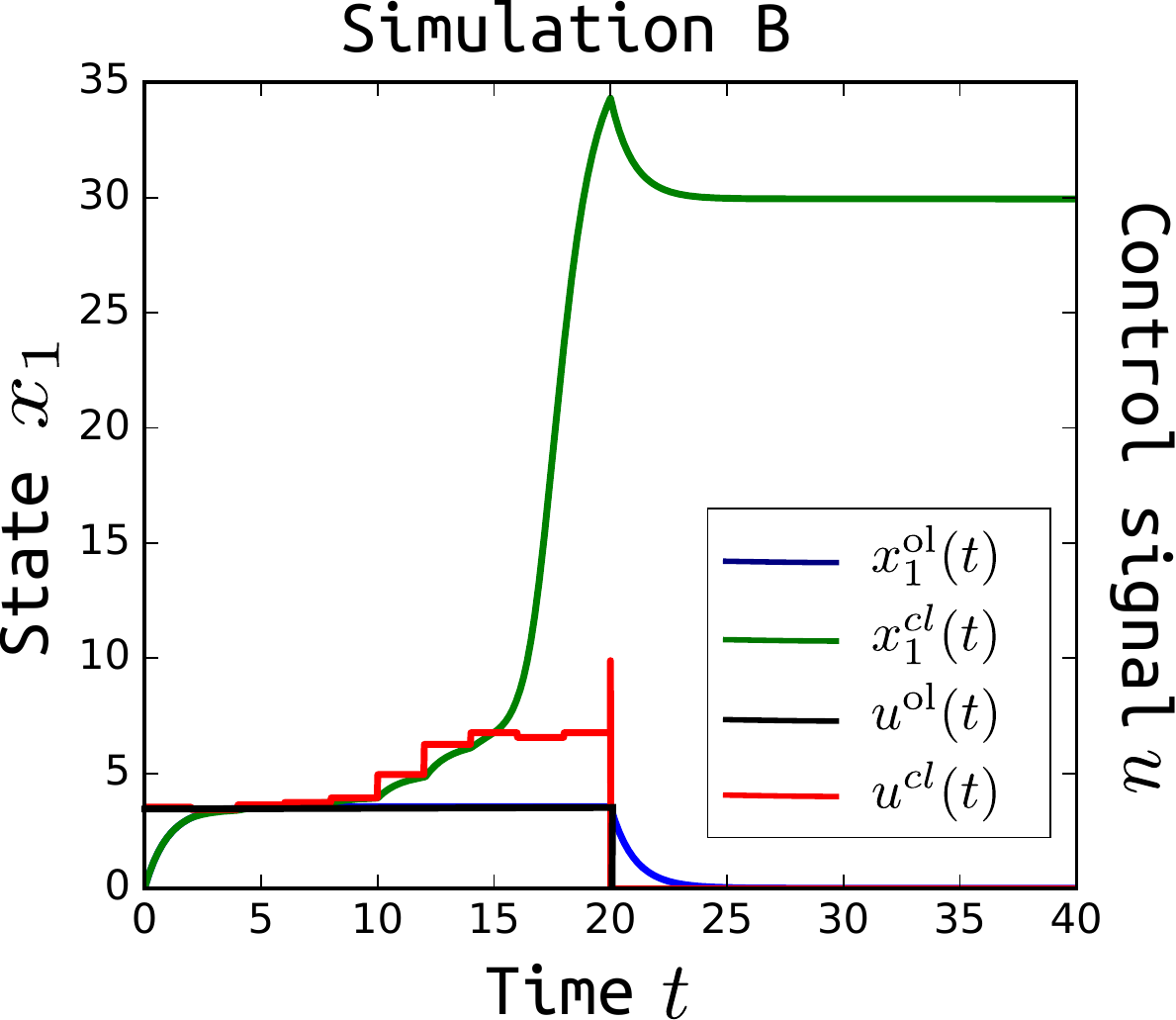}
	\caption{Closed-loop and open-loop switching. Left panel: level sets of $\cT_{\rm conv}$ with $\varepsilon = 10^{-2}$, level set $r = 0$, and energy budget curve $\mu \tau = \cE_{\rm max}$. Center and right panel: open and closed-loop simulations for Setting A and B, respectively. In both figures, $x_1^{\rm ol}$, $x_1^{\rm cl}$ stand for the trajectories of the state $x_1$ in the open and closed-loop settings, respectively, and $u^{\rm ol}$, $u^{\rm cl}$ stand for the corresponding control signals.}
	\label{fig:gr8-cl}
\end{figure*}

\section{Applications}\label{s:examples}
\subsection{Implementation Details}
We implemented our computational procedures both in python (using lsoda ordinary differential equation (ODE) solver) and Matlab (using ode15s ODE solver). We run our computational algorithm on a laptop equipped with a 4 core Intel i7 processor running at 2.4 GHz and 8 GB of RAM, however, we did not explicitly parallelize the computations. Computing one value of the function $r$ is equivalent to computing one trajectory of the system, albeit with high precision (we set relative tolerance of the solvers to $e^{-10}$ -- $e^{-14}$). 

\subsection{Closed-Loop Switching in Generalized Repressilator}
The eight species generalized repressilator is an academic example, where each of the species represses another species in a ring topology (cf.~\cite{Strelkowa10}). The corresponding dynamic equations for a symmetric generalized repressilator are as follows: 
\begin{align}
\label{eq:gr}\dot x_1 &= \frac{p_{1 1}^0}{1 + (x_{8}/p_{1 2}^0)^{p_{1 3}^0}} + p_{1 4}^0 - p_{1 5}^0 x_1 + u, \\
\notag\dot x_i &= \frac{p_{i 1}^0}{1 + (x_{i-1}/p_{i 2}^0)^{p_{i 3}^0}} + p_{i 4}^0 - p_{i 5}^0 x_i,~\forall i = 2,\dots 8,  
\end{align}
where $p_{i 1}^0 = 40$, $p_{i 2}^0 = 1$, $p_{i 3}^0 = 2$, $p_{i 4}^0 = 1$, and $p_{i 5}^0 = 1$. This system has two exponentially stable equilibria $x^\ast$ and $x^\bullet$ and is monotone with respect to the cones $\cK_x = P_x \R^8$ and $\cK_u = \R$, where $P_x= \diag([1,~-1,~1,~-1,~1,~-1,~1,~-1])$. We have also $x^\bullet \preceq_{\cK_x} x^\ast$. It can be shown that the unforced system is strongly monotone in the interior of $\Rnn^8$ for all positive parameter values. We consider here the problem of switching the system from one equilibrium $x^\bullet$ to the other equilibrium $x^\ast$ and we can verify that there exist pulse control signals that induce such a switch.

In the left panel of Figure~\ref{fig:gr8-cl}, we plot the level set of the function $\cT_{\rm conv}(x^\bullet, \mu,\tau, 10^{-2}) = \frac{1}{|\lambda_1|}\ln\left(\frac{r(x^\bullet, \mu, \tau)}{10^{-2}}\right) + \tau$, the level set $\cE_{\rm max} = 100$ of the function $\mu \tau = \cE_{\rm max}$, and the level set $r(x^\bullet,\mu,\tau) = 0$. The last two are related to the constraints of the static optimization program~\eqref{prog:static}. Note that $r$ is computed with the dominant eigenfunction associated with the target equilibrium $x^\ast$.  We also note that the function $\cT_{\rm conv}$ can escape to $-\infty$ around the level set $r(x^\bullet, \mu,\tau) \approx 0$. This is not a conflict with the interpretation of  the function $\cT_{\rm conv}$, since it represents the convergence time only if the value of $|r(x^\bullet,\mu,\tau)|$ is larger than $10^{-2}$. Otherwise the term $\frac{1}{|\lambda_1|}\ln\left(\frac{|r(x^\bullet, \mu, \tau)|}{10^{-2}}\right)$ is negative, and the computational results are meaningless. This also explains why the level sets of $\cT_{\rm conv}$ appear to have the same asymptotics as the level set $r(x^\bullet, \mu,\tau) \approx 0$ in Figure~\ref{fig:gr8-cl}. 

Our goal is to compare the open-loop (proposed by~\cite{sootla2015pulsesaut}) and closed-loop solutions to the switching problem subject to perturbations of parameters $p_{i j}^0$. We consider two settings for the simulation. In both settings, we compute the control signals based on the nominal model~\eqref{eq:gr} with the parameter values $p_{i j}^0$, but the simulations are obtained with two sets of (exact) parameter values:

\hspace{20pt}\emph{Setting A.}  We set $p_{i 1}^A = 50$ for odd $i$.

\hspace{20pt}\emph{Setting B.} We set $p_{i 1}^B = 30$ for odd $i$. 

The Euclidean distance between the nominal initial point and the actual initial point in Setting A and~B is equal to $0.025$ and $0.031$, respectively. In order to compute an open-loop optimal control policy based on the nominal model (i.e. with parameter values $p_{i j}^0$), one can solve the static optimization program~\eqref{prog:static}. The plots in Figure~\ref{fig:gr8-cl} also offer a graphical solution to the problem and a depiction of possible tradeoffs in the problem. In our case, the optimal solution lies at the intersection of the constraint curves (i.e. energy budget curve and level set $r(x^\bullet,\mu,\tau) = 0$). 

In our simulations, we pick a pair $(\mu^0, \tau^0)$ lying near the zero level set of $r$ below the level set $\cE_{\rm max} = 100$. This is not an optimal solution for the energy budget $\cE_{\rm max} = 100$, however, we pick a solution with a lower energy expenditure and a larger time $\tau$ (and hence larger convergence time) in order to have a possibility to react to the obtained measurements in the case of the closed-loop setting. We take $\tau^0 = 20$ and compute $\mu^0$ minimizing the time $\cT_{\rm conv}$ on a uniform grid of $100$ points in $[2,~10]$, which gives the value $\mu^0 = 3.53$.

For the closed-loop control, we take the same initial pair $(\tau^0, \mu^0)$.  In both setting~A and~B, we update the control signals every $t_{\rm samp} = 2$. For each update, we decrease the time of the pulse by $t_{\rm samp}$ and we decrease the available energy budget by subtracting the energy already consumed. We then compute the values of the function $r$ with a fixed $\tau = 20 - N t_{\rm samp}$, where $N$ is the number of previous updates and we choose the value $\mu$ on a uniform grid of $100$ points in $[2,~10]$, which minimizes $\cT_{\rm conv}$ with $r(x^\bullet, \mu, \tau) < 0$. 

The simulation results are depicted in the center and right panels of Figure~\eqref{fig:gr8-cl}. In Simulation A, the system converges to the target equilibrium faster than the nominal one (i.e, with parameters $p_{i j}^0$) and the closed-loop solution saves energy and limits the overshoot in comparison with the open-loop solution. In Simulation B, the opposite occurs and all the energy budget is spent. In this case, the closed-loop solution allows the switch, while the open-loop (i.e.~\cite{sootla2015pulsesaut}) does not. 

\subsection{Synchronization of Cardiac Cells}
\label{s:synchro}
Besides the switching problem, our approach can be used for more general problems. For example, a problem of importance in biology is to synchronize an ensemble of systems (e.g. cells) with a common control input. In this subsection, we propose to use the function $r$ in order to compute a train of temporal pulses and solve the synchronization problem.  Provided that the function $r$ and the state $x^j$ of each system are known, we design a pulse that drives every system toward isostables. The pair $(\mu^\ast,\tau^\ast)$ is chosen in such a way that the corresponding pulse minimizes the maximum time delay between the different systems, and it follows from~\eqref{time_isostable} that this pair can be computed as
\begin{equation*}
(\mu^\ast,\tau^\ast) = \argmin_{(\mu,\tau) \in \Rnn^2} \ln \left(\frac{\max_j r(x_j,\mu,\tau)}{\min_j r(x_j,\mu,\tau)}\right).
\end{equation*}
From a practical point of view, we note that the values of the function $r$ might be known only for some $(x,\mu,\tau)$ (especially if it is computed from data). In this case, the value $r$ at the different states $x_j$ of the systems can be computed by interpolation, and the optimal pair $(\mu^\ast,\tau^\ast)$ will be picked among the pairs $(\mu,\tau)$ for which the value of $r$ is known. 

\begin{algorithm}[t]
	\caption{Closed-loop synchronization of cells}
	\label{alg:synch}
	\begin{algorithmic}[1]
		\State {\bf Inputs:} $r$ for given $x$, $\tau$, $\mu$; time $T_p$ between two pulses; number of pulses $N_p$
		\For{$i= 1,\dots, N_p$}
		\State Observe the current states $x_j$ of the systems
		\State Compute the values of $r$ at $x_j$
		\State Find $(\mu^\ast,\tau^\ast)$
		\State Apply the pulse and wait (during $T_p-\tau$)
		\EndFor
	\end{algorithmic}
\end{algorithm}

We apply this closed-loop control (see Algorithm~\ref{alg:synch}) to synchronize FitzHugh-Nagumo systems (cf.~\cite{fitzhugh1961impulses,nagumo1962active}), which have been proposed as simple models of excitable cardiac cells. This example is motivated by the synchronization of cardiac cells (i.e. defibrillation) and is directly inspired by the study by~\cite{Wilson2014} proposing optimal defibrillation strategies. We consider here $100$ FitzHugh-Nagumo cells described by the dynamics (see also~\cite{Wilson2014})
\begin{align*}
\dot{V} & = 0.26 V(V-0.13)(1-V)-0.1 V w \\
\dot{w} & = 0.013 (V-w)
\end{align*}
where $V$ is the membrane potential and $w$ is a recovery (gating) variable. The values of $r$ were computed \emph{a priori} on a $20 \times 20$ grid for $(V,w) \in[0,2]\times [0,2]$ and on a $51 \times 41$ grid for $(\mu,\tau) \in [0,0.5] \times [10,50]$. The time between two successive pulses is $T_p=70$ and the initial conditions are randomly distributed on $[0,2]\times[0,2]$. The maximum time delay between the cells (computed with~\eqref{time_isostable}) after each pulse is shown in Figure~\ref{fig:pulse_synchro}(a) for the input obtained with the closed-loop control (optimal pairs $(\mu^\ast,\tau^\ast)$) and for periodic pulse trains. The best performance is obtained with the closed-loop control. As shown in Figure~\ref{fig:pulse_synchro}(b), the optimal pairs $(\mu^\ast,\tau^\ast)$ are not identical at each iteration, since they depend on the states of the cells, which motivates the use of closed-loop control. The first pulses correspond to a maximum value $\tau^*=50$, but $\mu^*$ takes intermediate values in the interval $[0,0.5]$. In particular, small values of $\mu$ are needed to obtain a fast convergence rate. We observe in Figure~\ref{fig:pulse_synchro}(a) that periodic pulse trains also synchronize the cells, but slower than our closed-loop approach. A periodic pulse train with maximum values $(\mu,\tau)=(0.5,50)$ yields a slow rate of convergence (red curve), while smaller values $(\mu,\tau)$ yield very large delays for the first iterations (green curve). Clearly, the optimal approach by~\cite{Wilson2014} outperforms our method in terms of convergence and time delays, however, our optimal control policy is easier to implement. Furthermore, we can parametrize our control signal with different (non-constant) basis functions.

\begin{figure}[t]\centering
\includegraphics[width = 0.85\columnwidth]{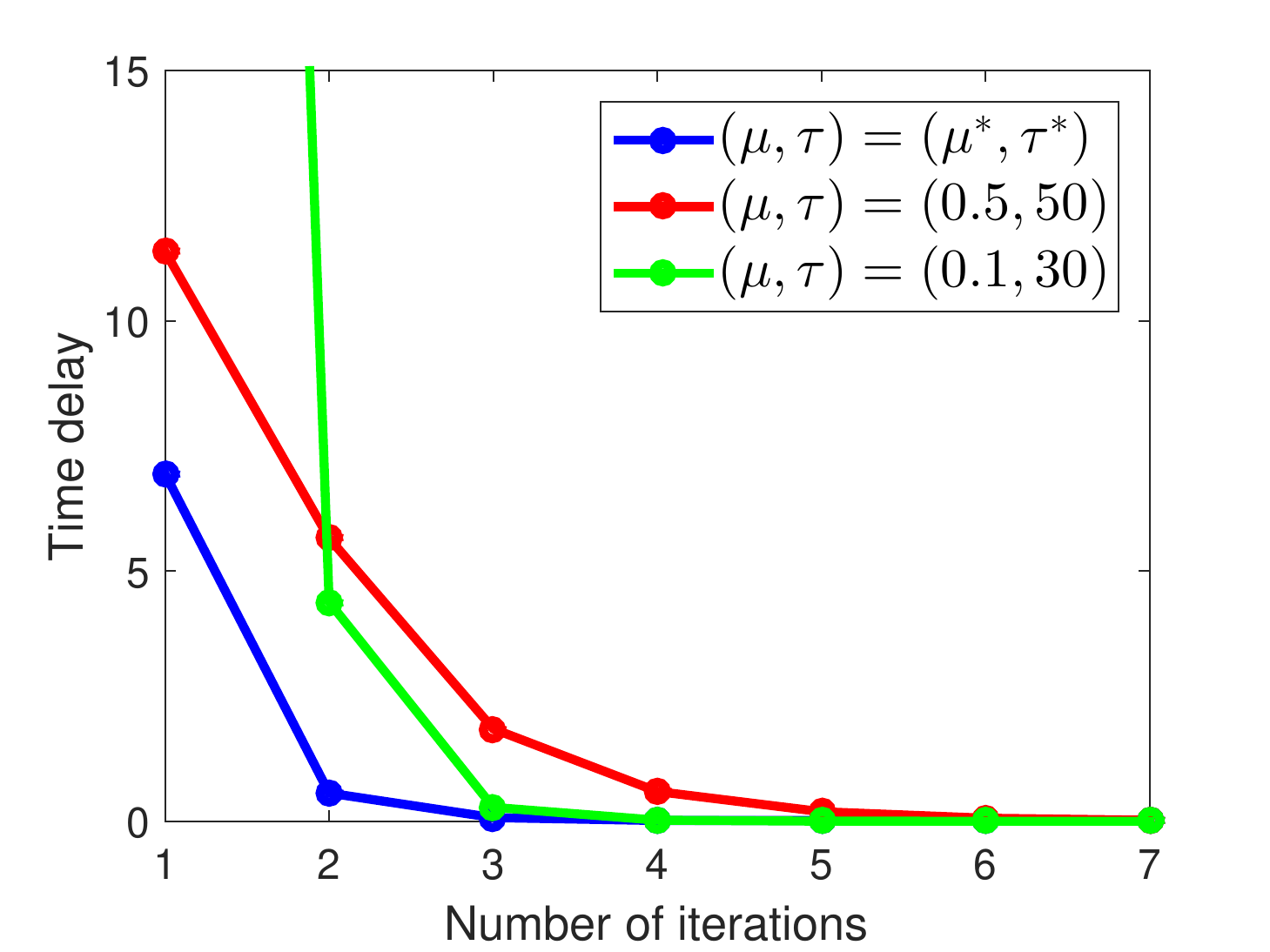}
\includegraphics[width = 0.85\columnwidth]{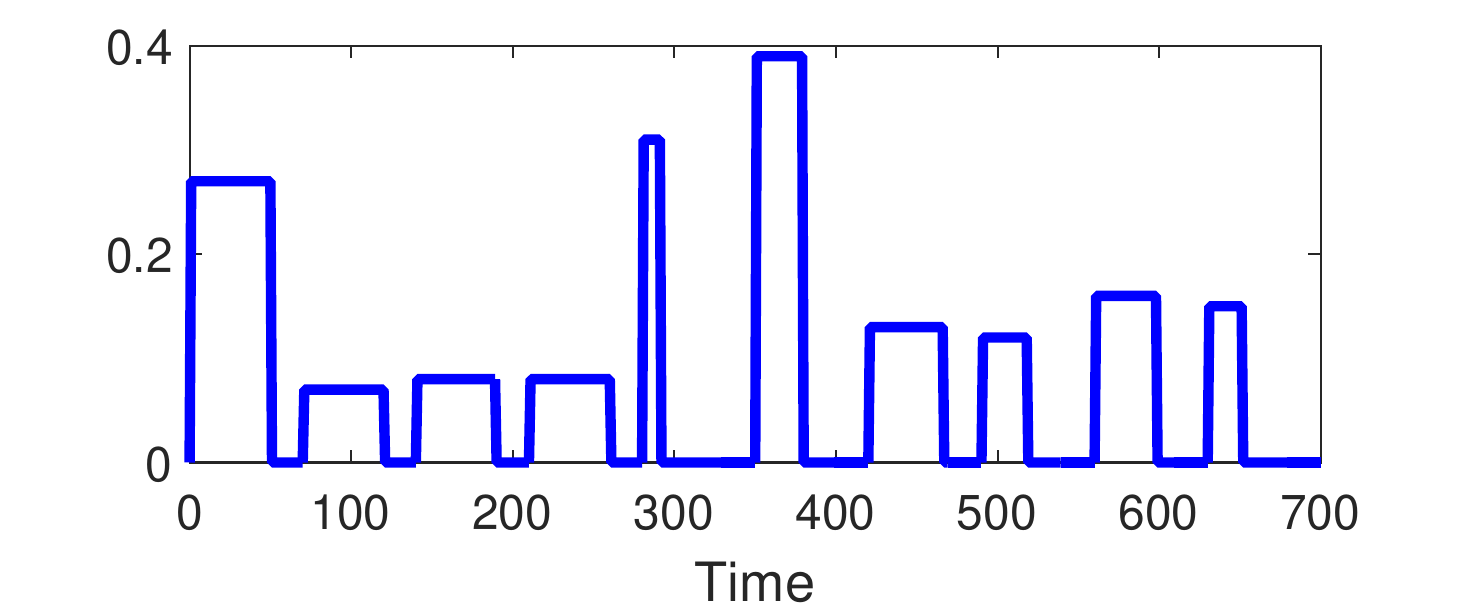}
\caption{\emph{Upper panel.} The maximum time delay between the FitzHugh-Nagumo cells decreases as the number of pulses increases and the cells eventually synchronize. The best performance is obtained with the closed-loop control based on the function $r$ (blue curve). Periodic pulses also synchronize the cells, but with a slower rate of convergence (red curve) or a large initial delay (green curve). 
	\emph{Lower panel.} The optimal pulse train consists of different optimal pairs $(\mu^\ast,\tau^\ast)$.} \label{fig:pulse_synchro}
\end{figure}

The proposed closed-loop control could be easily adapted to incorporate additional constraints (e.g. maximum energy) and provides a solution to the synchronization problem. This solution is a compromise between a simple but not optimal periodic pulse train and the complex exact solution of the optimal control problem established by~\cite{Wilson2014}. Future work could extend these preliminary results to more realistic cases, for instance where not all the states are observable.

\section{Conclusion}

In this paper, we studied a switching/convergence problem for monotone systems. Our solution reduces a dynamic optimization problem to the computation of the time-independent function $r$, which is defined using the Koopman operator. The properties of the function $r$ lead to straightforward solutions to a tradeoff between the convergence time and the energy budget. This approach can potentially be extended beyond monotone systems and switching/convergence problems. In this paper, we illustrate the possible benefits of a closed-loop solution for the switching problem. We also apply our framework to the synchronization of cardiac cells represented by non-monotone FitzHugh-Nagumo models. In this paper, we have not addressed partial state observability and/or partial controllability issues. This constitutes one of the future work directions.

\bibliography{Biblio}

\begin{thebibliography}{10}
\providecommand{\url}[1]{#1}
\csname url@rmstyle\endcsname
\providecommand{\newblock}{\relax}
\providecommand{\bibinfo}[2]{#2}
\providecommand\BIBentrySTDinterwordspacing{\spaceskip=0pt\relax}
\providecommand\BIBentryALTinterwordstretchfactor{4}
\providecommand\BIBentryALTinterwordspacing{\spaceskip=\fontdimen2\font plus
\BIBentryALTinterwordstretchfactor\fontdimen3\font minus
  \fontdimen4\font\relax}
\providecommand\BIBforeignlanguage[2]{{%
\expandafter\ifx\csname l@#1\endcsname\relax
\typeout{** WARNING: IEEEtran.bst: No hyphenation pattern has been}%
\typeout{** loaded for the language `#1'. Using the pattern for}%
\typeout{** the default language instead.}%
\else
\language=\csname l@#1\endcsname
\fi
#2}}

\bibitem{Purnick:2009}
P.~Purnick and R.~Weiss, ``{The second wave of synthetic biology: from modules
  to systems},'' \emph{Nat. Rev. Mol. Cell Biol.}, vol.~10, no.~6, pp.
  410--422, jun 2009.

\bibitem{brophy2014principles}
J.~Brophy and C.~Voigt, ``Principles of genetic circuit design,'' \emph{Nat
  methods}, vol.~11, no.~5, pp. 508--520, 2014.

\bibitem{milias2011silico}
A.~Milias-Argeitis, S.~Summers, J.~Stewart-Ornstein, I.~Zuleta, D.~Pincus,
  H.~El-Samad, M.~Khammash, and J.~Lygeros, ``In silico feedback for in vivo
  regulation of a gene expression circuit,'' \emph{Nat biotechnol}, vol.~29,
  no.~12, pp. 1114--1116, 2011.

\bibitem{Menolascina:2011}
F.~Menolascina, M.~Di~Bernardo, and D.~Di~Bernardo, ``{Analysis, design and
  implementation of a novel scheme for in-vivo control of synthetic gene
  regulatory networks},'' \emph{Automatica, Special Issue on Systems Biology},
  vol.~47, no.~6, pp. 1265--1270, Apr. 2011.

\bibitem{uhlendorf2012long}
J.~Uhlendorf, A.~Miermont, T.~Delaveau, G.~Charvin, F.~Fages, S.~Bottani,
  G.~Batt, and P.~Hersen, ``Long-term model predictive control of gene
  expression at the population and single-cell levels,'' \emph{Proc. Nat.
  Academy Sciences}, vol. 109, no.~35, pp. 14\,271--14\,276, 2012.

\bibitem{Gardner00}
T.~Gardner, C.~R. Cantor, and J.~J. Collins, ``Construction of a genetic toggle
  switch in escherichia coli,'' \emph{Nature}, vol. 403, pp. 339--342, 2000.

\bibitem{mezic2005}
I.~Mezi{\'c}, ``Spectral properties of dynamical systems, model reduction and
  decompositions,'' \emph{Nonlinear Dynam}, vol.~41, no. 1-3, pp. 309--325,
  2005.

\bibitem{mauroy2013isostables}
A.~Mauroy, I.~Mezi{\'c}, and J.~Moehlis, ``Isostables, isochrons, and koopman
  spectrum for the action--angle representation of stable fixed point
  dynamics,'' \emph{Physica D}, vol. 261, pp. 19--30, 2013.

\bibitem{fitzhugh1961impulses}
R.~FitzHugh, ``Impulses and physiological states in theoretical models of nerve
  membrane,'' \emph{Biophysical journal}, vol.~1, no.~6, p. 445, 1961.

\bibitem{nagumo1962active}
J.~Nagumo, S.~Arimoto, and S.~Yoshizawa, ``An active pulse transmission line
  simulating nerve axon,'' \emph{Proceedings of the IRE}, vol.~50, no.~10, pp.
  2061--2070, 1962.

\bibitem{sootla2015pulsesaut}
A.~Sootla, D.~Oyarz\'{u}n, D.~Angeli, and G.-B. Stan, ``Shaping pulses to
  control bistable systems analysis, computation and counterexamples,''
  \emph{Automatica}, vol.~63, pp. 254--264, Jan. 2016,
  http://arxiv.org/abs/1409.6150v3.

\bibitem{sootla2016nolcos}
A.~Sootla and A.~Mauroy, ``Shaping pulses to control monotone bistable systems
  using {K}oopman operator,'' in \emph{In Proc Symposium Nonlinear Control
  Systems}, Aug 2016, pp. 710--715.

\bibitem{mauroy2014converging}
A.~Mauroy, ``Converging to and escaping from the global equilibrium: Isostables
  and optimal control,'' in \emph{IEEE Conf Decision Control}, 2014, pp.
  5888--5893.

\bibitem{Wilson2014}
D.~Wilson and J.~Moehlis, ``An energy-optimal methodology for synchronization
  of excitable media,'' \emph{SIAM J Applied Dynamical Systems}, vol.~13,
  no.~2, pp. 944--957, 2014.

\bibitem{mezic2013analysis}
I.~Mezic, ``Analysis of fluid flows via spectral properties of the koopman
  operator,'' \emph{Annual Review of Fluid Mechanics}, vol.~45, pp. 357--378,
  2013.

\bibitem{mauroy2014global}
A.~Mauroy and I.~Mezi{\'c}, ``Global stability analysis using the
  eigenfunctions of the koopman operator,'' \emph{IEEE Trans Autom Control},
  vol.~61, no.~11, pp. 3356--3369, 2016.

\bibitem{Schmid2010}
P.~J. Schmid, ``Dynamic mode decomposition of numerical and experimental
  data,'' \emph{Journal of Fluid Mechanics}, vol. 656, pp. 5--28, 2010.

\bibitem{Tu2014}
J.~H. Tu, C.~W. Rowley, D.~M. Luchtenburg, S.~L. Brunton, and J.~N. Kutz, ``{On
  dynamic mode decomposition: Theory and applications},'' \emph{J Comput
  Dynamics}, vol.~1, no.~2, pp. 391 -- 421, December 2014.

\bibitem{angeli2003monotone}
D.~Angeli and E.~Sontag, ``Monotone control systems,'' \emph{IEEE Trans Autom
  Control}, vol.~48, no.~10, pp. 1684--1698, 2003.

\bibitem{sootla2016basins}
A.~Sootla and A.~Mauroy, ``Properties of isostables and basins of attraction of
  monotone systems,'' in \emph{Proc Am Control Conf}, 2016, pp. 7365--7370.

\bibitem{Strelkowa10}
N.~Strelkowa and M.~Barahona, ``{Switchable genetic oscillator operating in
  quasi-stable mode},'' \emph{J R Soc Interface}, vol.~7, no.~48, pp.
  1071--1082, 2010.

\bibitem{sootla2016geometry}
A.~Sootla and A.~Mauroy, ``Geometric properties and computation of isostables
  and basins of attraction of monotone systems,'' \emph{To appear in IEEE Trans
  Autom Control}, 2017, {https://arxiv.org/pdf/1705.02853}.

\bibitem{kim2016directed}
E.~S. Kim, M.~Arcak, and S.~A. Seshia, ``Directed specifications and assumption
  mining for monotone dynamical systems,'' in \emph{Proc Conf Hybrid Systems:
  Computation Control}.\hskip 1em plus 0.5em minus 0.4em\relax ACM, 2016, pp.
  21--30.

\bibitem{cataudella2013conditional}
I.~Cataudella, K.~Sneppen, K.~Gerdes, and N.~Mitarai, ``Conditional
  cooperativity of toxin-antitoxin regulation can mediate bistability between
  growth and dormancy,'' \emph{{PLoS} Comput Biol}, vol.~9, no.~8, p. e1003174,
  2013.

\bibitem{sootla2015koopman}
A.~Sootla and A.~Mauroy, ``Operator-theoretic characterization of eventually
  monotone systems,'' 2016, {http://arxiv.org/abs/1510.01149}.

\end{thebibliography}

\appendices
\section{Constant Control Signals are Optimal for a Minimum-Time Problem} \label{app:const-control}
Consider the following optimal control problem over bounded measurable control signals: 
\begin{align}
V(z, \mu, \beta) &= \inf\limits_{\begin{smallmatrix} \tau, u \in \cU_\infty([0,\mu]) \end{smallmatrix}}  \tau, \label{prob:opt-escape} \\
\notag \text{ subject to}&~\eqref{sys:f},\, x(0) = z,\, \\
\notag &x(\tau) \in \cC_\beta = \{y\in \R^n | s_1(y) = \beta \}, 
\end{align}
where  $s_1(x)$ is a $C^1$ increasing dominant eigenfunction defined on the basin of attraction of $x^\ast$. Under our assumptions, the solution to this problem is surprisingly straightforward.
\begin{prop}\label{thm:mon-val-fun}
	Let the system~\eqref{sys:f} satisfy Assumptions~{A1 -- A3}. Then \\	
	(i) If $s_1(z) < \beta$, then the optimal solution to~\eqref{prob:opt-escape}, if it exists, is $u(t) = \mu$ for all $t\in[0, \tau]$;\\	
	(ii) If $s_1(z) \ge \beta$, then the optimal solution to~\eqref{prob:opt-escape} is $u(t) = 0$ for all $t\ge0$. 	
	\end{prop} 
	\begin{proof}
		(i) Let $u^0(t) = \mu$ for all $t>0$, and $u^\delta(t)$ be any admissible control signal, then $u^0(t) \succeq u^\delta(t)$ for all $t\in[0,~\tau]$. Then by monotonicity we have $\phi(t, z, u^0(\cdot))\succeq \phi(t, z, u^\delta(\cdot))$, which leads to $s_1(\phi(t, z, u^0(\cdot)))\ge s_1(\phi(t, z, u^\delta(\cdot)))$ for all $t\ge 0$ due to Proposition~\ref{prop:mon-eig-fun}.  Hence
		\begin{align*}
		\beta &= s_1(\phi(\tau, z, u^0(\cdot)))\ge  s_1(\phi(\tau, z, u^\delta(\cdot))),\\
		\beta &> s_1(\phi(t, z, u^0(\cdot))) \ge  s_1(\phi(t, z, u^\delta(\cdot))) \textrm{ for } t<\tau,
		\end{align*}
		which implies that the target set $\cC_\beta^-$ is reached with $u^0(\cdot)$ at least as fast as with any other admissible control signal $u^\delta(\cdot)$. Therefore, the control signal $u^0(t) = \mu$ is an optimal solution of the problem.
		
		(ii) The proof is similar to the point (i). 
		\end{proof}
This result justifies our use of temporal pulses to solve convergence problems for monotone systems. The problem~\eqref{prob:opt-escape} has a direct relation to the function $r$. In particular, $s_1(x)<\beta$ implies that $r(x, \mu, V(x, \mu, \beta)) = \beta$ provided that the problem~\eqref{prob:opt-escape} has a solution. On the other hand, there might be some values $\tau \ne V(x, \mu, \beta)$ such that  $r(x, \mu, \tau) = \beta$. In general, if $s_1(x)<\beta$, then 
\begin{gather}
V(x, \mu, \beta) = \min \{\tau \in \Rp | r(x,\mu,\tau) = \beta\},
\end{gather}
provided that the solution exists. 
In particular, if the premise of the points (ii) and (iii) in Lemma~\ref{lem:r-prop} holds then $V(x, \mu, \beta) = \tau$ if and only if $r(x,\mu,\tau) = \beta$, which again justifies our use of the pulse control function in these cases.
\section{Level Sets of the Function $r$} \label{app:additional-res}
If the premise of the points (ii) and (iii) in Lemma~\ref{lem:r-prop} hold, then the level sets of $r$ are the graphs of strictly decreasing functions, a result which simplifies their computations as discussed in~\cite{sootla2016nolcos}. The algorithms describing this procedure were developed by~\cite{sootla2016basins, sootla2016geometry} and~\cite{kim2016directed}.

\begin{cor} \label{cor:mon-switch-iso-mon-sys}
	Let the system~\eqref{sys:f} satisfy Assumptions~{A1}, {A4}, {A6}, {A7}, and $\alpha > 0$. \\
	(i) If $s_1(x) < -\alpha$, then $\{(\mu,\tau) \in \Rp^2| r(x, \mu,\tau) = -\alpha \}$ is a graph of a strictly decreasing function, i.e., this set does not contain pairs $(\mu_1, \tau_1)\ne (\mu_2, \tau_2)$ such that $\mu_1\le \mu_2$ and $\tau_1 \le \tau_2$;\\
	(ii) If $s_1(x) < \alpha$ and $f(x,\eta)\succeq 0$, then $\{(\mu,\tau) \in \Rp^2 | r(x, \mu,\tau) = \alpha \}$ is a graph of a strictly decreasing function for $\mu>\eta$.
	\end{cor}
	\begin{proof}(i) First, we make sure that the function $r(x,\cdot,\cdot)$ is increasing.  This is guaranteed if $\partial_\mu r(x,\mu,\tau)$, $\partial_\tau r(x,\mu,\tau)$ are positive. \\		
		Due to Lemma~\ref{lem:r-prop} we have that $\partial_\tau r(x, \mu, \tau) >0$ as long as $r(x, \mu, \tau)\le 0$, and that $\partial_\mu r(x,\mu,\tau)>0$.
		Now since $\partial_\tau r(x, \mu, \tau) >0$ (for all $(\mu,\tau)$ such that $r(x,\mu,\tau)=-\alpha<0$), the implicit function theorem implies that there exists a function $\tau = g(\mu, \alpha)$ such that $r(x, \mu, g(\mu, \alpha)) = -\alpha$. Furthermore, the function $g$ is $C^1$ in $\mu$ and $\partial_\mu g = - \partial_\mu r(x, \mu, \tau) / \partial_\tau r(x, \mu, \tau) < 0$ in the neighborhood of the level set $\{ (\mu, \tau)  | r(x, \mu, \tau) =-\alpha\}$. Therefore the level set $\{\tau, \mu | r(x, \mu, \tau) =-\alpha\}$ is a graph of a strictly decreasing function in $\mu$. It also directly follows that  the level set $\{ (\mu , \tau) | r(x, \mu, \tau) =-\alpha\}$ is a graph of a strictly decreasing function in $\tau$.
		
		(ii) We have that $\partial_\mu r(x,\mu,\tau) > 0$. Now positivity of $\partial_\tau r(x,\mu,\tau)$ for positive $r(x,\mu,\tau)$ follows from the point (iii) in Lemma~\ref{lem:r-prop}. The rest of the proof follows on similar lines as in (i).
		\end{proof}
\section{Purely Data-Based Control Setting}\label{app:DMD}
One can compute $r$ from sampled data points (e.g. obtained from experiments). If data points
 \begin{multline}
 \label{eq:time_series}
 z[k]=\phi(\tau+k T_s, x, \mu h(\cdot, \tau)) \\= \phi(k T_s,\phi(\tau, x, \mu),0)
 \end{multline}
 are given (with the sampling time $T_s$), then the DMD algorithm can be used to compute $s_1(\phi(\tau, x, \mu ))=r(x,\mu,\tau)$. We illustrate this in the following subsections.
\subsection{Dynamic mode decomposition algorithm} 
The dynamic mode decomposition algorithm can be used to estimate the eigenfunctions of the Koopman operator from data. Assume that $N$ snapshots of $m$ trajectories of the (unforced) system are given, i.e.
\begin{equation*}
z^j[k] = g(\phi(k T_s,x_0^j,0))\,,   j=1,\dots,m\,, k=1,\dots,N
\end{equation*}
where $T_S$ is the sampling time, $x_0^j \in \mathbb{R}^n$ is the initial condition, and $g$ is an observable (typically measuring one of the states). The algorithm yields the so-called DMD modes and eigenvalues, which are related to the eigenfunctions and eigenvalues of the Koopman operator. It is described in Algorithm~\ref{alg:dmd}. We will not elaborate further on the details of the method and refer the reader to~\cite{Schmid2010,Tu2014}. We only note that the matrix $Y X^{\dagger}$, where $X^{\dagger}$ is the pseudo-inverse of $X$, is a finite-dimensional approximation of the Koopman semigroup. Therefore the eigenvectors of $Y X^{\dagger}$ can be used to estimate the eigenfunctions of the Koopman operator.

\begin{algorithm}[t]
	\caption{Dynamic Mode Decomposition}
	\label{alg:dmd}
	\begin{algorithmic}[1]
		\State {\bf Input:} Data points $z^j[k]$,  $j=1,\dots,m$, $k=1,\dots,N$;
		\State {\bf Output:} DMD modes $V_l$ and associated DMD eigenvalues $\nu_l$ ($l=1,\dots,m$); Koopman eigenfunctions $s_l(x_0^j)$ and eigenvalues $\lambda_l$;
		\State Construct the matrices
		\begin{equation*}
		\begin{split}
		& X=\left[ \begin{array}{ccc} z^1[1] & \cdots &  z^1[N-1]\\
		\vdots & & \vdots \\
		z^m[1] & \cdots &  z^m[N-1] \end{array} \right]\in \mathbb{R}^{m \times (N-1)} \\
		& Y=\left[ \begin{array}{ccc} z^1[2] & \cdots &  z^1[N]\\
		\vdots & & \vdots \\
		z^m[2] & \cdots &  z^m[N] \end{array} \right]\in \mathbb{R}^{m \times (N-1)} \,;
		\end{split}
		\end{equation*}
		\State Compute the reduced singular value decomposition of $X$, i.e.
		\begin{equation*}
		X = U \Sigma V^* \,;
		\end{equation*}
		\State Construct the matrix
		\begin{equation*}
		T = U^* Y V \Sigma^{-1} \,;
		\end{equation*}
		\State Compute the eigenvectors $w_l$ and eigenvalues $\nu_l$ of $T$, i.e.
		\begin{equation*}
		T w_l = \nu_l w_l \,;
		\end{equation*}
		\State Compute the DMD modes $V_l = U w_l$ (the associated DMD eigenvalues are $\nu_l$);
		\State The Koopman eigenfunction $s_l$ at $x_0^j$ is given by the $j$th component of the DMD mode $V_l$, and the corresponding Koopman eigenvalue is given by $\lambda_l=\ln(\nu_l)/T_s$.
	\end{algorithmic}
\end{algorithm}
\subsection{Application of the DMD Algorithm to Toxin-antitoxin System} \label{ss:dmd}
\begin{figure}[t]\centering
	\includegraphics[width = 0.7\columnwidth]{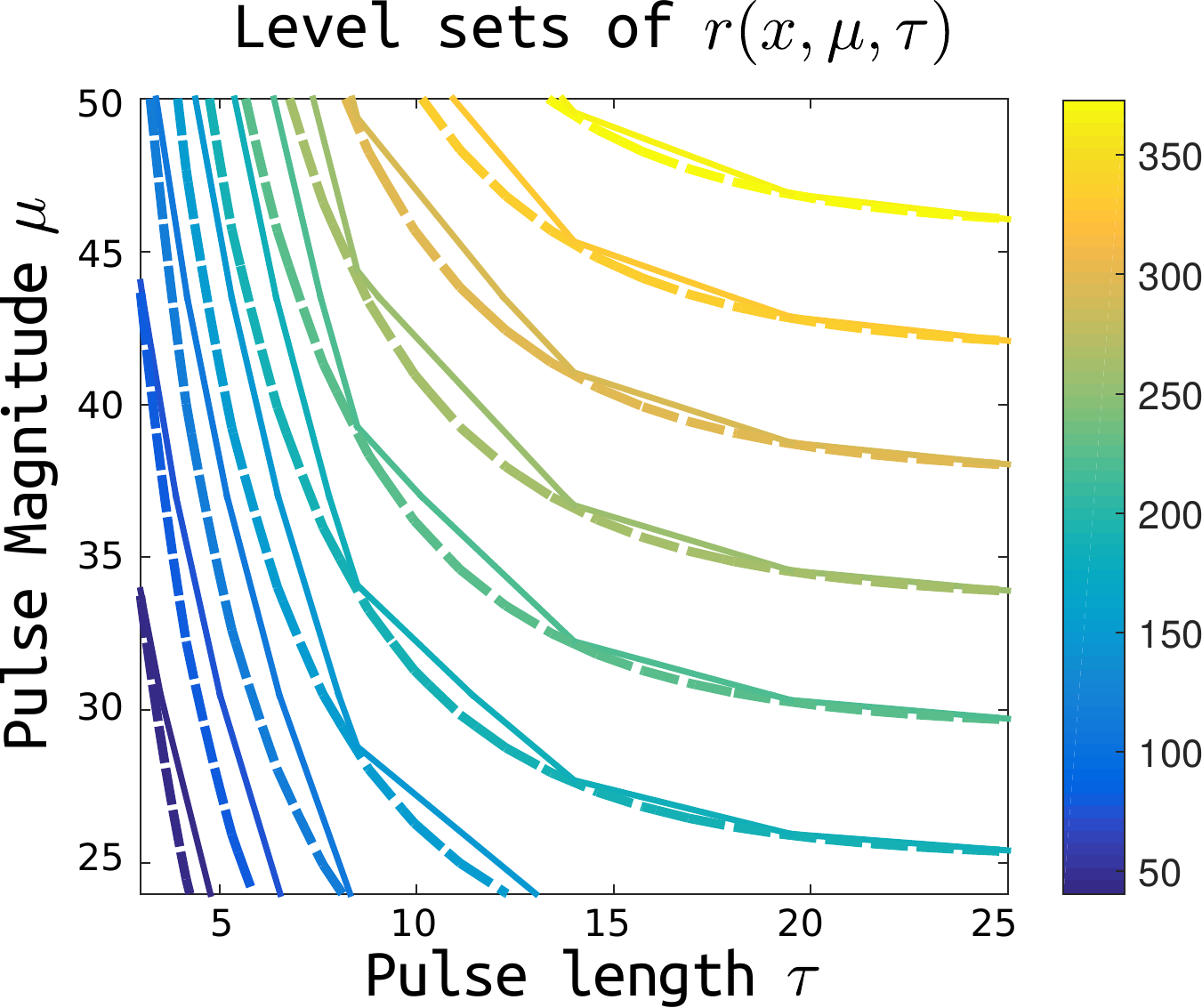}
	\caption{Level sets of $r$ computed with the DMD algorithm (solid lines) and with Laplace averages (dashed lines).}\label{switching_DMD}
\end{figure}
Consider the toxin-antitoxin system studied by~\cite{cataudella2013conditional}. We will use the model and parameter values as described in~\cite{sootla2015koopman}. 

The system is bistable with two exponentially stable equilibria $x^\bullet$ and $x^\ast$, but not monotone with respect to any orthant. However, it was established by~\cite{sootla2015koopman} that it is \emph{eventually monotone}, i.e. the flow satisfies the monotonicity property after some initial transient. The level sets of the function $r(x^\bullet, \mu, \tau)$ (in the space of the parameters $(\mu, \tau)$) were computed by~\cite{sootla2016nolcos} and it appears that these sets are monotone curves although the system is not monotone. 

Since we cannot guarantee that the function $r$ is increasing, we need to develop a different computational approach applicable to a broader class of systems. One possibility is to compute the function $r$ from data. To do so, we used the DMD algorithm with $25$ time series of the form~\eqref{eq:time_series} (each of which corresponds to a different pair $(\mu,\tau) \in [24,50] \times [3,25]$). For each time series, we used only $4$ snapshots over the time interval $[\tau+12.5, \tau+50]$ (i.e. $k \in \{1,2,3,4\}$ with $T_s=12.5$ in~\eqref{eq:time_series}). The level sets of $r$ computed with the DMD algorithm are accurate and similar to the ones obtained with Laplace averages (Figure~\ref{switching_DMD}). We stress that this example is only a proof-of-concept, and further research is required to automate the application of DMD to this problem.

\end{document}